\documentclass[a4j,12pt]{amsart}
\usepackage{amsmath,amssymb,amsthm,enumerate,
}
\usepackage[all]{xy}
\usepackage{setspace}
\usepackage{indentfirst}
\newtheorem{thm}{Theorem}[section]
\newtheorem{lemma}[thm]{Lemma}
\newtheorem{prop}[thm]{Proposition}
\newtheorem{cor}[thm]{Corollary}
\theoremstyle{definition}
\newtheorem{rem}[thm]{Remark}
\newtheorem{conj}[thm]{Conjecture}
\newtheorem{define}[thm]{Definition}

\newtheorem{define-thm}[thm]{Definition-Theorem}
\usepackage{amssymb}
\usepackage[all]{xy}

\newcommand{\Hom}{{\rm Hom}}

\newcommand{\End}{{\rm End}}
\newcommand{\Ind}{^{\uparrow{\tilde{G}}}}

\newcommand{\tilt}{{\rm tilt}\,}
\newcommand{\silt}{{\rm silt}\,}
\newcommand{\hmtB}{K^b({\rm proj}\,B)}

\newcommand{\hmttildeB}{K^b({\rm proj}\,\tilde{B})}

\makeatletter
\@namedef{subjclassname@2020}{%
  \textup{2020} Mathematics Subject Classification}
\makeatother
\begin{document}

\title{
On tilting complexes over blocks covering cyclic blocks}
\author{Yuta KOZAKAI}
\address[Yuta KOZAKAI]{Department of Mathematics,  
Tokyo University of Science}
\email{kozakai@rs.tus.ac.jp}


\maketitle
\begin{abstract}
Let $p$ be a prime number, $k$ an algebraically closed field of characteristic $p$,
$\tilde{G}$ a finite group, and $G$ a normal subgroup of $\tilde{G}$
having a $p$-power index in $\tilde{G}$.
Moreover let $B$ be a block of $kG$ with a cyclic defect group
and $\tilde{B}$ be the unique block of $k\tilde{G}$ covering $B$.
We study tilting complexes over the block $\tilde{B}$ and show that
the block $\tilde{B}$ is a tilting-discrete algebra.
Moreover we show that the set of all tilting complexes over $\tilde{B}$
is isomorphic to that over $B$ as partially ordered sets.
\end{abstract}

\section{Introduction}
In representation theory of finite groups,
there is a well-known and important conjecture
called Brou\'{e}'s abelian defect group conjecture.
\begin{conj}
Let $k$ be an algebraically closed field of characteristic $p>0$,
$G$ a finite group,
$B$ a block of the group algebra $kG$ with defect group $D$,
and $b$ the Brauer correspondent of $B$ in $kN_G(D)$.
If $D$ is abelian, then the block $B$ is derived equivalent to $b$.
\end{conj}

There are many cases that the
Brou\'{e}'s abelian defect group conjecture holds,
for example
$D$ is a cyclic group \cite{Ri1989-2, Rou1993},
$D$ is a Klein four group \cite{E1990, Ri1996},
$D$ is a $p$-solvable group \cite{HL2000}
and more.

It is known that Brou\'{e}'s abelian defect group conjecture
does not hold generally without the assumption that
the defect group $D$ is abelian.
However even if the defect group $D$ is not abelian,
it is thought that the similar statement holds in some situations
and that how we may state the nonabelian version conjecture.
On the other hand, in general, a derived equivalence between blocks induces a
perfect isometry 
(see \cite[Proposition 4.10]{B1994} or \cite[Theorem 9.2.9]{1685Rickard}).
Hence it is expected that two blocks should be derived equivalent
if there is a perfect isometry between the blocks,
here we recall that an isometry $I: \mathbb{Z}{\rm Irr}(B)\rightarrow 
\mathbb{Z}{\rm Irr}(b)$ is called a perfect isomery if it satisfies
separability condition and integrality condition 
(further details can be seen in \cite[D\'{e}finition(2) 1.1]{B1990}).
In regards to this, we know the following result.

\begin{thm}[{\cite[Theorem 1.1 (iv)]{HKK2010}}]\label{HKK-Theorem1.1}
Let $k$ be an algebraically closed field of characteristic $p>2$.
Let $\tilde{G}$ be a finite group with a Sylow $p$-subgroup 
$\tilde{P}$ isomorphic to $C_{p^n}\rtimes C_p$ for an integer $n\geq 2$.
Let $\tilde{B}_0$ be the principal block of $k\tilde{G}$ and
$\tilde{b}_0$ the principal block of $kN_{\tilde{G}}(P)$,
where $P$ is a subgroup of $\tilde{P}$ isomorphic to $C_{p^n}$.
Then there exists an isometry 
$\tilde{I}: \mathbb{Z}{\rm Irr}(\tilde{B}_0) 
\rightarrow \mathbb{Z}{\rm Irr}(\tilde{b}_0)$
satisfying the separability condition.
\end{thm}

In addition to the notation in Theorem \ref{HKK-Theorem1.1},
let $G$ be a normal subgroup such that
its index in $\tilde{G}$ is $p$ and that $G$ has
a Sylow $p$-subgroup isomorphic to $C_{p^n}$ 
(the existence of such a finite group is ensured by
\cite[Theorem 1.1 (i)]{HKK2010}).
Moreover, let $B_0$ be the principal block of $kG$ 
and $b_0$ the principal block of $kN_G(P)$.
The isometry $\tilde{I}$ 
in Theorem \ref{HKK-Theorem1.1} is constructed
from the perfect isometry $I$ between $B_0$ and $b_0$ naturally,
so we can strongly expect that the isometry 
$\tilde{I}$ is in fact a perfect isometry.
In particular, 
our interests are perfect isometries 
between $\tilde{B}_0$ and $\tilde{b}_0$ obtained by the ones between
$B_0$ and $b_0$ coming from the derived equivalence
between $B_0$ and $b_0$.
Indeed we know various derived equivalences
between $B_0$ and $b_0$ by \cite{Ri1989-2,RS2002,Ri1996,Koz-Ku2018,Koz2019}
because $B_0$ and $b_0$ have cyclic
groups isomorphic to $C_{p^n}$ as defect groups, which means 
they are Brauer tree algebras with the same number of 
simple modules and the same multiplicity (see \cite{Al86}).
Moreover it may be said that 
there are derived equivalences between $\tilde{B}_0$ and $\tilde{b}_0$
induced from the ones between $B_0$ and $b_0$
and that the perfect isometries are in fact induced by
the derived equivalences between $\tilde{B}_0$ and $\tilde{b}_0$.
From these,
one of our ambitions is showing that the block $\tilde{B}_0$ is
derived equivalent to $\tilde{b}_0$ by using the one between
$B_0$ and $b_0$ we know a lot of examples of.
For the proof, it is essential 
to find a tilting complex over $\tilde{B}_0$ such that
its endomorphism algebra in the homotopy category
is Morita equivalent to $\tilde{b}_0$ by \cite[Theorem 6.4]{Ri1989-1}.
Hence it is important to classify the
tilting complexes over $\tilde{B}_0$ for our ambitions.

After that, in \cite{AI2012} Aihara-Iyama introduced  a
partial order on the set of all silting complexes
over $\Lambda$, where $\Lambda$ means a finite dimensional algebra
(see Definition \ref{parial-order}).
Moreover, if $\Lambda$ is a symmetric algebra,
then any silting complex over $\Lambda$ is 
a tilting complex over $\Lambda$ (see Proposition \ref{silting-tilting}).
Hence the set of all tilting complexes over the block $\tilde{B}$
has the partially ordered structure because
block algebras of finite groups are symmetric algebras.
Therefore our first aim is to give a classification of
tilting complexes over the block $\tilde{B}_0$.
On the other hand, in \cite{AI2012} they introduced 
silting mutations which are
operations producing other silting complexes from the ones,
and defined silting mutation quivers (Definition \ref{silting-quiver}).
In that paper, they show that 
the silting mutation quiver is equal to 
the Hasse quiver of all silting complexes on the partial order above
\cite[Theorem 2.35]{AI2012}.
Another aim of this paper is to give a further development of
the mutation theory of silting objects.
Starting with \cite{AI2012}, 
the mutation theory of silting objects
in the derived categories over finite-dimensional algebras
has been studied by many researchers,
but there are few studies of the theory for the modular representation theory.
Hence we aim at investigating the silting objects of the derived categories
over certain block algebras.
In particular, we consider the following two questions:
\begin{itemize}
\item
When is a block algebra silting(tilting)-discrete algebra
(Definition \ref{def-tilting-discrete})?
\item
When do two block algebras have
the same silting (tilting) mutation quiver?
\end{itemize}

The first question is one of the themes of silting theory,
and many researchers gave various examples of
silting-discrete algebras, which are characterized by the property that
all silting complexes are compact up to equivalence,
for example
Brauer graph algebras whose Brauer graphs have at most one
odd cycle and no cycle of even length \cite{AAC2018},
representation-finite piecewise hereditary algebras \cite{Ai2013},
representation-finite symmetric algebras \cite{Ai2013},
preprojective algebras of Dynkin type \cite{AiM2017},
derived-discrete algebras with finite global dimension \cite{BPP2016}
and
algebras of dihedral, semidihedral and quaternion type \cite{EJR2018} and more.
If we get a positive answer of the first question
and if we get a silting-discrete block, hence tilting-discrete block,
then we understand that all tilting complexes are connected
to each other by iterated mutation and this enables us to
understand of the derived category of
the modules over the block (see Theorem \ref{strongly-connected}).

Also the answer of the second question lets us enable to
understand the tilting complexes over a block algebra
from those over the other one.
Hence we may reduce the consideration of the easier block
to find a suitable tilting complex
when we want to find a suitable tilting complex.
Thus it makes sense to consider the above two questions.

On these questions, Koshio and the author gave a 
partial answer in \cite{KK2020}.
In the paper,
it is shown that the number of two-term tilting complexes
over $\tilde{B}_0$ is finite
and that the induction functor gives
the isomorphism as partially ordered sets
between two-term tilting complexes over
$B_0$ and the ones over $\tilde{B}_0$ by using $\tau$-tilting theory
introduced by Adachi-Iyama-Reiten \cite{AIR2013},
here two-term tilting complexes
means tilting complexes with vanishing cohomologies
in  degrees other than $0$ and $-1$.

\begin{thm}[{\cite[Theorem 1.2]{KK2020}}]\label{KK2020-thm}
Let $\tilde{G}$ be a finite group,
$G$ a normal subgroup such that 
the index $|\tilde{G} : G|$ is a $p$-power,
$k$ an algebraically closed field of characteristic $p>0$,
$B$ a block of $kG$, and
$\tilde{B}$ the unique block of $k\tilde{G}$
covering $B$.
Assume $\tilde{B}$ satisfies the following conditions:
\begin{enumerate}[(i)]
\item
Any indecomposable $B$-module is $I_{\tilde{G}}(B)$-invariant.
\item
$B$ is a $\tau$-tilting finite algebra,
that is $B$ has a finite number of two-term tilting complexes.
\end{enumerate}
Then
the induction functor $(-)^{\uparrow \tilde{G}}: {\rm mod}\:B
\rightarrow {\rm mod}\:\tilde{B}$ gives 
an isomorphism between the set of two-term tilting complexes
over $B$ and that of $\tilde{B}$
as the partially ordered sets.
\end{thm}
Here, we remark that in the setting of Theorem \ref{HKK-Theorem1.1},
the blocks $B_0$ and $\tilde{B}_0$ satisfy the condition 
in Theorem \ref{KK2020-thm}.
Moreover, we remark that if an algebra is a tilting-discrete algebra,
then the number of two-term tilting complexes over the algebra is finite.
Thus the above theorem is a partial solution of the questions and
would be helpful for the solutions for the above questions.
In this paper we generalize Theorem \ref{KK2020-thm}
and give the following result
under the more general assumption
than the case in Theorem \ref{HKK-Theorem1.1}
in conjunction with giving positive answers of above questions.

\begin{thm}\label{mainthm1}(see Theorem \ref
{tilting-discrete} and Theorem \ref{poset-iso})
Let $\tilde{G}$ be a finite group,
$G$ a normal subgroup such that 
the index $|\tilde{G} : G|$ is a $p$-power,
$k$ an algebraically closed field of characteristic $p>0$,
$\tilde{B}$ a block of $k\tilde{G}$,
and $B$ the unique block of $kG$ covered by $\tilde{B}$.
Assume $\tilde{B}$ satisfies the following conditions:
\begin{enumerate}[(i)]
\item
Any indecomposable $B$-module is $I_{\tilde{G}}(B)$-invariant.
\item
$B$ is a tilting-discrete algebra.
\item
Any algebra derived equivalent to $B$
has a finite number of two-term tilting complexes.
\end{enumerate}
Then $\tilde{B}$ is a tilting-discrete algebra.
Moreover the induction functor $(-)\Ind : \hmtB
\rightarrow \hmttildeB $
induces an isomorphism 
between $\tilt B$ and $\tilt \tilde{B}$
as partially ordered sets,
here $\tilt B$ and $\tilt \tilde{B}$ mean the
set of all tilting complexes over $B$ and $\tilde{B}$
respectively. 
\end{thm}

Now we return to the case in which we are interested,
that is, the covered block $B$ of $kG$ has a cyclic defect group.
Then the conditions (i), (ii) and (iii) of
Theorem \ref{mainthm1} are satisfied automatically (see Proposition \ref{cyclic-block}),
so we can apply Theorem \ref{mainthm1} to the block $B$.
Moreover, in the setting of Theorem \ref{mainthm1},
since the principal block $B_0$ of $kG$
is the unique block covered by the
principal block $\tilde{B}_0$ of $k\tilde{G}$,
the following theorem would be helpful for the consideration for 
the situation of Theorem \ref{HKK-Theorem1.1}.

\begin{thm}(see Theorem \ref{mainthm})
Let $\tilde{G}$ be a finite group having $G$
as a normal subgroup
with its index in $\tilde{G}$ a $p$-power.
Let $B$ be a block of the finite group $G$ with cyclic defect group
and $\tilde{B}$ the unique block of $k\tilde{G}$ covering $B$.
Then the following hold.
\begin{enumerate}
\item
$\tilde{B}$ is a tilting-discrete algebra.
\item
The induction functor $(-)\Ind : \tilt B \rightarrow \tilt \tilde{B}$
induces an isomorphism of partially ordered sets.
\end{enumerate}
\end{thm}

%

In this paper, we use the following notation and terminology.

Here, modules are finitely generated right modules unless otherwise stated.
For a finite group $G$, a subgroup $H$ of $G$
and a $kH$-module $U$,
we denote by $U^{\uparrow G}$ the induced module $U \otimes_{kH} kG$.

For a finite dimensional algebra $\Lambda$,
we denote by $|\Lambda|$ the number of 
the isomorphism classes of simple $\Lambda$-modules.
We denote by ${\rm mod}\:\Lambda$ the category of finitely generated
right $\Lambda$-modules,
and by ${\rm proj}\:\Lambda$ the category of finitely generated
projective $\Lambda$-modules.
We denote by $K^b({\rm proj}\:\Lambda)$
the bounded homotopy category of ${\rm proj}\:\Lambda$.
For $C \in K^b({\rm proj}\:\Lambda)$,
we denote by ${\rm add}\:C$
the smallest full subcategory of $K^b({\rm proj}\:\Lambda)$
which contains $C$
and is closed under finite co-products, summands
and isomorphisms.

This paper is organized as follows.
In Section \ref{Preliminaries}
we state several notation and results on
block theory and silting theory.
In Section \ref{mainsection}
we give the proof of the main theorems.

\section{Preliminaries}\label{Preliminaries}


\subsection{Block theory}
In this section, let $k$ be an algebraically closed field
of characteristic $p>0$.
We denote by $G$ a finite group,
and by $k_G$ the trivial module of $kG$,
that is, a one-dimensional vector space on which
each element in $G$
acts as the identity.
We recall the definition of blocks of group algebras.
The group algebra $kG$ has a unique decomposition
$$ kG=B_1\times \cdots \times B_n $$
into a direct product of subalgebras $B_i$
each of which is indecomposable as an algebra.
Then each direct product component $B_i$
is called a block of $kG$.
For any indecomposble $kG$-module $M$,
there exists a unique block $B_i$ of $kG$
such that $M=MB_i$ and $MB_j=0$ for all $j \in \{1, \ldots, n\}-\{i\}$.
Then we say that $M$ lies in the block $B_i$ or
that $M$ is a $B_i$-module.
Also we denote by $B_0(G)$ the principal block of $kG$,
that is, the unique block of $kG$ which does not annihilate
the trivial $kG$-module $k_G$.


First, we recall the definition of defect groups of blocks of
finite group algebras and their properties.

\begin{define}
Let $B$ be a block of $kG$.
A minimal subgroup $D$ of $G$
which satisfies the following condition is uniquely determined
up to conjugacy in $G$:
the $B$-bimodule epimorphism
$$B\otimes_{kD} B \rightarrow B\:\:(b_1\otimes_{kD} b_2 \mapsto b_1b_2)$$
is a split epimorphism.
We call the subgroup a defect group of the block $B$.
\end{define}

The following results are well known (for example, see \cite{Al86}).
\begin{prop}
For the principal block $B_0(G)$ of $kG$,
its defect group is a Sylow $p$-subgroup of $G$.
\end{prop}

\begin{prop}\label{cyclic-defect}
For a block $B$ of $kG$ and a defect group $D$ of $B$,
the following are equivalent:
\begin{enumerate}[(1)]
\item
$D$ is a non-trivial cyclic group;
\item
$B$ is of finite representation type and is not semisimple;
\item
$B$ is a Brauer tree algebra.
\end{enumerate}
\end{prop}

\begin{define}
Let $H$ be a subgroup of $G$.
For a $kH$-module $U$, we denote by $U^{\uparrow G}:=U\otimes_{kH}kG$
the induced module of $U$ from $H$ to $G$.
Also, for a complex $X=(X^i, d^i)$,
we denote by $X^{\uparrow G}$ the complex $(X^i\otimes_{kH}kG, d^i\otimes_{kH}kG)$.
This induces a functor from $K^b({\rm proj}\:kH)$
to $K^b({\rm proj}\:kG)$.
\end{define}


We will consider the  case where $G$ is a normal subgroup of 
a finite group $\tilde{G}$
and has a $p$-power index.
The following is
the complex version of Green's indecomposability theorem.

\begin{prop}[see {\cite{G1959}}]\label{Green}
Let $G$ be a normal subgroup of
a finite group $\tilde{G}$ of a $p$-power index.
If $X$ is a bounded indecomposable complex of $kG$-modules,
then the induced complex $X^{\uparrow \tilde{G}}$
of $k\tilde{G}$-modules is also a bounded indecomposable 
complex of $k\tilde{G}$-modules.
\end{prop}

\begin{proof}
The proof of \cite[Theorem 8.8]{Al86} works for complexes,
hence we get the result.
\end{proof}

\begin{prop}[{\cite[Corollary 5.5.6]{NT1989}}]\label{cover}
Let $G$ be a normal subgroup of $\tilde{G}$,
and $B$ a block of $G$.
If the index of $G$ in $\tilde{G}$ is a $p$-power,
then there exists a unique block of $k\tilde{G}$ covering $B$.
\end{prop}
\begin{rem}\label{induced-complex}
Let $G$ be a normal subgroup of a finite group
$\tilde{G}$ of a $p$-power index,
$B$ a block of $kG$, and $\tilde{B}$ the unique block
of $k\tilde{G}$ covering $B$.
Then by Propositions \ref{Green} and \ref{cover},
for any indecomposable complex $X$ of $K^b({\rm proj}\:B)$,
we can easily show that
the induced complex $X^{\uparrow \tilde{G}}$
is an indecomposable complex
of $K^b({\rm proj}\:\tilde{B})$.
\end{rem}


\subsection{Silting mutations}
In this section, $\Lambda$ means a
finite-dimensional algebra unless otherwise stated.
We say that a complex $P$ of $K^b({\rm proj}\:\Lambda)$
is basic if $P$ is isomorphic to a direct sum
of indecomposable complexes which are mutually
non-isomorphic.
\begin{define}
Let $P$ be a complex of $K^b({\rm proj}\:\Lambda)$.
We say that the complex is a silting complex
(or tilting complex, respectively)
if the following conditions are satisfied:
\begin{enumerate}[(1)]
\item
$\Hom_{K^b({\rm proj}\:\Lambda)}(P, P[n])=0$ for any 
$n>0$ (for any $n \neq 0$, respectively).
\item
It holds that ${\rm thick}\,P = K^b({\rm proj}\:\Lambda)$,
where ${\rm thick}\,P$ is the smallest thick subcategory containing ${\rm add}\:P$.
\end{enumerate}
\end{define}

By the above definition,
it holds that tilting complexes are silting complexes,
but the converse does not hold generally.
By the following proposition, the silting complexes over 
symmetric algebras are tilting complexes.
In particular, silting complexes over block algebras are tilting complexes.

\begin{prop}[{\cite[Example 2.8]{AI2012}}]\label{silting-tilting}
If $\Lambda$ is a finite-dimensional symmetric algebra,
then any silting complex over $\Lambda$ is
a tilting complex.
\end{prop}

In \cite[Definition 2.10, Theorem 2.11]{AI2012},
it is shown that there is a partial order on the set of
silting complexes.
\begin{define}\label{parial-order}
Let $P$ and $Q$ be silting complexes of $K^b({\rm proj}\:\Lambda)$.
We define a relation $\geq$ between $P$ and $Q$ as follows; 
$$P \geq Q :\Leftrightarrow \Hom_{K^b({\rm proj}\:\Lambda)}(P, Q[i])=0
\:(^\forall i >0).$$
Then the relation $\geq$ gives a partial order on $\silt\,\Lambda$,
where $\silt\,\Lambda$ means the set of isomorphism classes of 
basic silting complexes over $\Lambda$.
\end{define}

We recall the definition of mutations for silting complexes of $K^b({\rm proj}\:\Lambda)$
\cite[Definition 2.30, Theorem 3.1]{AI2012}.

\begin{define}
Let $P$ be a basic silting complex of $K^b({\rm proj}\:\Lambda)$
and decompose it as $P=X \oplus M$.
We take a triangle
$$ X \xrightarrow{f} M' \rightarrow Y \rightarrow $$
with a minimal left $({\rm add}\,M)$-approximation $f$ of $X$.
Then the complex $\mu_X^{-}(P):=Y\oplus M$ is 
a silting complex in $K^b({\rm proj}\:\Lambda)$ again.
We call the complex $\mu_X^{-}(P)$ a left mutation of $P$ with respect to $X$.
If $X$ is indecomposable, then we say that the left mutation 
is irreducible.
We define the (irreducible) right mutation $\mu^+_X(P)$ dually.
Mutation will mean either left or right mutation.
\end{define}

\begin{rem}
If $\Lambda$ is a finite-dimensional symmetric algebra,
then, for any tilting complex $P=X\oplus M$ over $\Lambda$,
the complex $\mu^{\epsilon}_X(P)$ is a tilting complex again
by Proposition \ref{silting-tilting},
where $\epsilon$ means $+$ or $-$.
\end{rem}


The silting mutation quiver was introduced by \cite[Definition 2.41]{AI2012}.
\begin{define}[{\cite[DEFINITION 2.41]{AI2012}}]\label{silting-quiver}
The silting mutation quiver of $K^b({\rm proj}\:\Lambda)$
is defined as follows:
\begin{itemize}
\item
The set of vertices is $\silt\,\Lambda$.
\item
We draw an arrow from $P$ to $Q$ if $Q$ is an irreducible left mutation of $P$. 
\end{itemize}
\end{define}


The following proposition shows that 
the Hasse quiver 
of the partially ordered set $\silt\,\Lambda$
is exactly the silting mutation quiver of $K^b({\rm proj}\:\Lambda)$.

\begin{thm}[{\cite[Theorem 2.35]{AI2012}}]\label{Theorem 2.35}
For any silting complexes $P$ and $Q$ over $\Lambda$,
the following conditions are equivalent:
\begin{enumerate}[(1)]
\item
$Q$ is an irreducible left mutation of $P$;
\item
$P$ is an irreducible right mutation of $Q$;
\item
$P>Q$ and there is no silting complex $L$
satisfying $P>L>Q$.
\end{enumerate}
\end{thm}


From now on, we assume that $\Lambda$ is a
finite-dimensional symmetric algebra unless otherwise stated.
In particular, any silting complex over $\Lambda$
is a tilting complex over $\Lambda$ by Proposition \ref{silting-tilting}.


We recall the definition of tilting-discrete algebras.
\begin{define}\label{def-tilting-discrete}
We say that an algebra (which is not necessarily a symmetric algebra) 
$\Lambda$ is a tilting-discrete algebra
if for all $\ell >0$ and any tilting complex $P$ over $\Lambda$,
the set 
$$ \{ T\in \tilt\,\Lambda\:|\: P\geq T \geq P[\ell] \} $$
is a finite set,
where $\tilt\,\Lambda$ means the set of isomorphism classes 
of basic tilting complexes over $\Lambda$.
\end{define}


\begin{thm}[{\cite[Theorem 3.5]{Ai2013}}]\label{strongly-connected}
If $\Lambda$ is a tilting-discrete algebra,
then $\Lambda$ is a strongly tilting connected algebra,
that is,
for any tilting complexes $T$ and $U$,
the complex $T$ can be obtained from $U$
by either iterated irreducible left mutation
or iterated irreducible right mutation.
\end{thm}


There is an equivalent condition on the tilting-discreteness,
which plays an important role later.
\begin{thm}[{\cite[Theorem 1.2]{AiM2017}\label{AiM2017}}]
For a finite-dimensional self-injective algebra $\Lambda$,
the following are equivalent.
\begin{enumerate}[(1)]
\item $\Lambda$ is a tilting-discrete algebra.
\item $2$-${\rm tilt}_P\,\Lambda:=\{ T\in \tilt \Lambda\:|\:
P \geq T \geq P[1] \}$
 is a finite set for any tilting complex
$P$ which is given by iterated irreducible left tilting mutation from $\Lambda$.
\end{enumerate}
\end{thm}

The following lemma is essential for the proof 
of Theorem \ref{tilting-discrete}.

\begin{lemma}\label{graphlemma}
Let $\Lambda$ be a finite-dimensional symmetric algebra
and $P$ a tilting complex over $\Lambda$.
Then the underlying graph of the Hasse quiver of the partially ordered set
$2$-${\rm tilt}_P\,\Lambda$ is a
$|\Lambda|$-regular graph.
Moreover, for any algebra $\Gamma$ derived equivalent to $\Lambda$,
if the number of two-term tilting complexes over $\Gamma$ is finite,
then $2$-${\rm tilt}_P\,\Lambda$ is a finite $|\Lambda|$-regular graph.
\end{lemma}

\begin{proof}
First, we remark that, for any symmetric finite-dimensional algebra $\Lambda$
and for any tilting complex $T$ over $\Lambda$,
if it holds that $\Lambda \geq T \geq \Lambda[1]$, then $T$ is a two-term tilting
complex by the definition of the partial order on $\tilt \:\Lambda$
(we recall that a complex $X=(X^i)$ is said to be a two-term complex
if $X^i=0$ if $i\neq 0, -1$).
Let $P$ be a tilting complex over $\Lambda$.
Since finite-dimensional symmetric algebras are closed under the derived equivalence,
the endomorphism algebra $\End_{K^b({\rm proj}\:\Lambda)}(P)$
is also a finite-dimensional symmetric algebra.
Now, we consider the derived equivalence 
$F: K^{-}({\rm proj}\:\Lambda) \rightarrow 
K^{-}({\rm proj}\:\End_{K^b({\rm proj}\:\Lambda)}(P))$
induced by the tilting complex $P$ over $\Lambda$.
The functor induces an isomorphism between the following two partially ordered sets;
\begin{itemize}
\item
$2\textnormal{-}{\rm tilt}_P\,\Lambda = \{ T\in \tilt \Lambda\:|\:
P \geq T \geq P[1] \}$
\item
$
\{ T'\in \tilt \:\End_{K^b({\rm proj}\:\Lambda)}(P)\:|\:
\End_{K^b({\rm proj}\:\Lambda)}(P) \geq T' \geq \End_{K^b({\rm proj}\:\Lambda)}(P)[1] \}.
$
\end{itemize}
The latter is the set of all two-term tilting complexes over $\End_{K^b({\rm proj}\:\Lambda)}(P)$.
Hence it is isomorphic to the support $\tau$-tilting quiver
of $\End_{K^b({\rm proj}\:\Lambda)}(P)$
by \cite[Corollary 3.9]{AIR2013},
so its underlying graph
is a $|\Lambda|$-regular graph
since the derived equivalence between
$\Lambda$ and $\End_{K^b({\rm proj}\:\Lambda)}(P)$
means $|\End_{K^b({\rm proj}\:\Lambda)}(P)|=|\Lambda|$
(we recall that two support $\tau$-tilting modules
are connected by an arrow in the support $\tau$-tilting
quiver if and only if they are mutations of each other
and that for each support $\tau$-tilting module $M$ over $\Lambda$
there are $|\Lambda|$ sorts of mutations of $M$).
Therefore the underlying graph of the Hasse quiver of the set
$2\textnormal{-}{\rm tilt}_P\,\Lambda$
is a $|\Lambda|$-regular graph too.

The remaining argument is clear because
$\End_{K^b({\rm proj}\:\Lambda)}(P)$ is derived equivalent to $\Lambda$
for any tilting complex $P$ over $\Lambda$.
\end{proof}


\begin{cor}\label{graphlemma-for-Brauer-tree-algebra}
For a Brauer tree algebra $\Lambda$ and any tilting complex $P$ over $\Lambda$,
the set $2$-${\rm tilt}_P\,\Lambda$ is a finite $|\Lambda|$-regular graph.
\end{cor}
\begin{proof}
If $\Lambda$ is a Brauer tree algebra,
then, for any tilting complex $P$ over $\Lambda$,
the endomorphism algebra $\End_{K^b({\rm proj}\:\Lambda)}(P)$ 
is a Brauer tree algebra again.
Hence the number of elements in $2\textnormal{-}{\rm tilt}_P\,\Lambda$
which is isomorphic to the set of all two-term tilting complex over
$\End_{K^b({\rm proj}\:\Lambda)}(P)$
is finite
by \cite[Theorem 1.1]{AMN2020} or 
\cite[Theorem 5.2]{Ai2013}.
Hence we can apply Lemma \ref{graphlemma} and complete the proof.
\end{proof}


\section{Main Theorem and its Proof}\label{mainsection}
\subsection{Notation and Assumption}\label{notation}
In this section, we use the following notation.
Let $k$ be an algebraically closed field of characteristic $p>0$.
We denote by $G$ a finite group,
and by $\tilde{G}$ a finite group having $G$ as a normal subgroup
and assume that the index of $G$ in $\tilde{G}$ is $p^n$ for some positive integer $n$.
We denote by $B$ a block of $kG$.
Then by Proposition \ref{cover}, 
there exists a unique block of $k\tilde{G}$ covering $B$.
We denote the unique block of $k\tilde{G}$ by $\tilde{B}$.
Also 
we denote by $I_{\tilde{G}}(B):=\{\tilde{g}\in \tilde{G}\:|\: \tilde{g}B\tilde{g}^{-1}=B \}$
the inertial group of the block $B$ in $\tilde{G}$.
Moreover we assume the following condition is satisfied:
\begin{enumerate}
\item[(Inv)]
Any $B$-module is $I_{\tilde{G}}(B)$-invariant,
that is, for any $B$-module $U$ and any $\tilde{g} \in I_{\tilde{G}}(B)$,
it holds that $U\tilde{g} \cong U$ as $B$-modules.
\end{enumerate}

\begin{lemma}\label{Inv-lemma}
The condition \textnormal{(Inv)} on the invariance of modules
implies that of complexes of $\hmtB$.
\end{lemma}

\begin{proof}
Assume that the condition (Inv) is satisfied
and that any complex of $\hmtB$ of length less than or equal to $n$ is $I_{\tilde{G}}(B)$-invariant.
Take an arbitrary complex 
$(X_0 \xrightarrow{d_0} X_1 \xrightarrow{d_1} X_2 \xrightarrow{d_2}\cdots 
\xrightarrow{d_{n-1}}X_n)$
of projective $B$-modules,
and let $Y$ be the truncated complex 
$(X_1 \xrightarrow{d_1} X_2\xrightarrow{d_2}\cdots \xrightarrow{d_{n-1}}X_n)$.
Then, by the assumption we have that
$$Y\tilde{g} =
(X_1\tilde{g} \xrightarrow{d'_1} X_2\tilde{g} \xrightarrow{d'_2}
\cdots \xrightarrow{d'_{n-1}}X_n\tilde{g})
\cong Y$$
for any $\tilde{g}\in I_{\tilde{G}}(B)$,
where $d'_i: X_i\tilde{g}\rightarrow X_{i+1}\tilde{g}$
is the homomorphism which maps $x_i\tilde{g}$ to $d_i(x_i)\tilde{g}$
for each $i$.
We need to show that 
$$X\tilde{g} = (X_0\tilde{g} \xrightarrow{d'_0} X_1\tilde{g} 
\xrightarrow{d'_1} X_2\tilde{g} \xrightarrow{d'_2} 
\cdots \xrightarrow{d'_{n-1}}X_n\tilde{g})
 \cong X.$$
We easily see that ${\rm Im}\,d'_i = ({\rm Im}\,d_i)\tilde{g}$ for each $i$.
Hence
we have that ${\rm Im}\,d_0 \cong {\rm Im}\,d'_0$
by the condition (Inv).
By the injectivity of $I({\rm Im}\,d_0)$,
we have the extension $\alpha_0': I({\rm Im}\,d_0) \rightarrow I({\rm Im}\,d'_0)$
of the isomorphism ${\rm Im}\,d_0\cong {\rm Im}\,d'_0$, where $I({\rm Im}\,d_0)$ and $I({\rm Im}\,d'_0)$
mean  injective envelopes of $I({\rm Im}\,d_0)$ and $I({\rm Im}\,d'_0)$,
respectively, that is we have the following commutative diagram:
\[
\xymatrix{
 {\rm Im}\,d_0 \ar@{>->}[r]^{i_0}\ar[d]^{\cong} & I({\rm Im}\,d_0) \ar@{.>}[ddl]^{\alpha'_1} \\
 {\rm Im}\,d'_0 \ar@{>->}[d] & \\
 I({\rm Im}\,d'_0)  
.}
\]
Since the homomorphism $\alpha_1'\circ i_0$ is a monomorphism
and $i_0:  {\rm Im}\,d_0 \rightarrow I({\rm Im}\,d_0)$ is an essential monomorphism,
we have the homomorphism $\alpha_1'$ is a monomorphism.
Hence we have that $\alpha_1'$ is an isomorphism because
$\dim  I({\rm Im}\,d_0)  = \dim I({\rm Im}\,d'_0)$.
Since $I({\rm Im}\,d_0) $ is a direct summand of $X_1$
and $I({\rm Im}\,d'_0)$ is that of $X_1\tilde{g}$, 
we have an isomorphism 
$\alpha''_1 :X_1/I({\rm Im}\,d_0) \rightarrow X_1\tilde{g}/I({\rm Im}\,d'_0)$
by Krull-Schmidt theorem.
Using two direct decompositions
$X_1= I({\rm Im}\,d_0) \oplus X_1/I({\rm Im}\,d_0)$
and 
$X_1\tilde{g}=I({\rm Im}\,d'_0) \oplus X_1\tilde{g}/I({\rm Im}\,d'_0)$,
we define the homomorphism
$$ \alpha_1: X_1= I({\rm Im}\,d_0) \oplus X_1/I({\rm Im}\,d_0)
\xrightarrow{
\begin{bmatrix}
\alpha_1' &0 \\ 0& \alpha''_1
\end{bmatrix}}
 I({\rm Im}\,d'_0) \oplus X_1\tilde{g}/I({\rm Im}\,d'_0)
 =X_1\tilde{g}.
$$
Obviously the homomorphism $\alpha_1$ is an isomorphism.

Next, let $P({\rm Im}\,d_0)$ and $P({\rm Im}\,d'_0)$ be
projective covers of ${\rm Im}\,d_0$ and ${\rm Im}\,d'_0$, respectively.
By the projectivity of $P({\rm Im}\,d_0)$,
we have a homomorphism $\alpha_0'$ which makes the
following diagram commutative:
\[
\xymatrix{
& P({\rm Im}\,d_0) \ar@{.>}[ddl]_{\alpha'_0}\ar@{->>}[d] \\
& {\rm Im}\,d_0 \ar[d]_{\cong}\\
P({\rm Im}\,d'_0)\ar@{->>}[r]& {\rm Im}\,d'_0 
.}
\]
The dual argument to the previous one gives the homomorphism
$\alpha'_0$  is an isomorphism.
Hence we have an isomorphism 
$$ 
\alpha_0: X_0= P({\rm Im}\,d_0) \oplus X_0/P({\rm Im}\,d_0)
\xrightarrow{
\begin{bmatrix}
\alpha_0' &0 \\ 0& \alpha''_0
\end{bmatrix}}
P({\rm Im}\,d'_0) \oplus X_0\tilde{g}/P({\rm Im}\,d'_0)
 =X_0\tilde{g},
$$
where $\alpha_0'': X_0/P({\rm Im}\,d_0) \rightarrow 
X_0\tilde{g}/P({\rm Im}\,d'_0)$
is an isomorphism.
By the construction of $\alpha_0$ and $\alpha_1$,
we have the following commutative diagram:
\[
\xymatrix{
X_0 \ar[r]_{d_0}\ar[d]_{\alpha_0}^{\cong} & X_1 \ar[d]_{\alpha_1}^{\cong}\\
X_0\tilde{g} \ar[r]_{d_0'} & X_1\tilde{g}
}
\]
Moreover we can take an injective resolution $I(Y)$ 
starting with ${\rm \alpha_1}: X_1\rightarrow X_1\tilde{g}$
since $X_1\tilde{g}$ is an injective envelope of $X_1$:
\[
\xymatrix{
Y: &X_1 \ar[d]^{\alpha_1}\ar[r]^{d_1} & X_2 \ar[r]^{d_2}\ar[d]^{\beta_2}
&\cdots \ar[r]^{d_{n-1}} &X_n\ar[d]^{\beta_n}  \\
I(Y): &X_1\tilde{g}\ar[r] & I(Y)_2 \ar[r] &\cdots \ar[r]& I(Y)_n.
 }
\]
Since $\beta_2 \circ d_1 \circ \alpha_1^{-1}\circ d'_0
=0$, we have a complex
$$ X_0\tilde{g} \xrightarrow{d_0'} X_1\tilde{g} \rightarrow I(Y)_2\rightarrow
\cdots \rightarrow I(Y)_n.$$
Also, by the assumption that $Y \cong Y\tilde{g}$,
we have that $I(Y) \cong Y\tilde{g} $,
that is we have the following commutative diagram
with each $\gamma_i$ invertible:
\[
\xymatrix{
I(Y): &X_1\tilde{g}\ar@{=}[d]\ar[r] & I(Y)_2 \ar[r]\ar[d]^{\gamma_2} &\cdots \ar[r]& I(Y)_n\ar[d]^{\gamma_n}\\
Y\tilde{g}: &X_1\tilde{g} \ar[r]^{d'_1} & X_2\tilde{g} \ar[r]^{d'_2}
&\cdots \ar[r]^{d'_{n-1}} &X_n\tilde{g}.
 }
\]
Therefore we have the following commutative diagram:
\[
\xymatrix{
X: &X_0 \ar[d]^{\alpha_0}\ar[r]^{d_0}&X_1 \ar[d]^{\alpha_1}\ar[r]^{d_1} & X_2 \ar[r]^{d_2}\ar[d]^{\beta_2}
&\cdots \ar[r]^{d_{n-1}} &X_n\ar[d]^{\beta_n}  \\
&X_0\tilde{g}\ar@{=}[d]\ar[r]^{d_0'} &X_1\tilde{g} \ar@{=}[d]\ar[r] &  I(Y)_2 \ar[r]\ar[d]^{\gamma_2}
&\cdots \ar[r] &I(Y)_n\ar[d]^{\gamma_n}  \\
X\tilde{g}: &X_0\tilde{g}\ar[r]^{d_0'} &X_1\tilde{g}\ar[r]^{d'_1} & X_2\tilde{g}\ar[r]^{d'_2} &\cdots \ar[r]^{d'_{n-1}}& X_n\tilde{g}.
 }
\]
Since all vertical homomorphisms are invertible,
we get that $X\cong X\tilde{g}$.

\end{proof}

\begin{lemma}\label{approximation}
Let $X$ and $Y$ be complexes of $\hmtB$,
and let $\mathcal{M}$ be a full subcategory of $\hmtB$.
Then the following hold.
\begin{enumerate}[(1)]
\item
For a minimal left $\mathcal{M}$-approximation $f: X \rightarrow M$ of $X$ in $\hmtB$,
the induced map $f\Ind: X\Ind \rightarrow M\Ind$ is a minimal left 
$\mathcal{M}\Ind$-approximation of $X\Ind$ in $\hmttildeB$.
\item
For a minimal right $\mathcal{M}$-approximation $f: M \rightarrow Y$ of $Y$ in $\hmtB$,
the induced map $f\Ind: M\Ind \rightarrow Y\Ind$ is a  minimal right 
$\mathcal{M}\Ind$-approximation of $Y\Ind$ in $\hmttildeB$.
\end{enumerate}
\end{lemma}
\begin{proof}
We prove (1), and the other can be proved similarly.

First we show that 
$f\Ind: X\Ind \rightarrow M\Ind$ is left minimal.
Take a morphism $h : M\Ind \rightarrow M\Ind$
satisfying $h\circ f\Ind =f\Ind$.
By restricting $h\circ f\Ind =f\Ind$ to $X\otimes_{kG} 1\cong X$,
we have that $h\circ f = f$, which means that
$h : M\rightarrow M$ is an isomorphism by the left minimality of $f$.
Hence $h : M\Ind \rightarrow M\Ind$ is an isomorphism too.

It remains to show that
$f\Ind: X\Ind \rightarrow M\Ind$ is a left 
$\mathcal{M}\Ind$-approximation of $X\Ind$ in $\hmttildeB$.
We can assume $I_{\tilde{G}}(B)=\tilde{G}$ because
the unique block of $kI_{\tilde{G}}(B)$ covering $B$ is Morita equivalent
to the block $\tilde{B}$ of $k\tilde{G}$ by \cite[Theorem 5.5.12]{NT1989}.
Let $Z$ be a complex of $\mathcal{M}$,
and let $\alpha : X\Ind \rightarrow Z\Ind$
be a morphism in $\hmttildeB$.
We show that there exists a morphism $\beta : M\Ind \rightarrow Z\Ind$
such that $\alpha=\beta\circ f\Ind$,
which means that $f\Ind$ is a left $\mathcal{M}\Ind$-approximation of $X\Ind$
in $\hmttildeB$.
By Frobenius reciprocity and Mackey's formula,
we get isomorphisms
\begin{align*}
\Hom_{\hmttildeB}(X\Ind, Z\Ind)
&\cong
\Hom_{\hmtB}(X, Z\Ind{}_{\downarrow G})\\
&\cong
\bigoplus_{\tilde{g}G\in \tilde{G}/G}\Hom_{\hmtB}(X,  Z\tilde{g})\\
&\cong
\Hom_{\hmtB}(X, Z)^{\oplus p^n},
\end{align*}
here the last isomorphism comes from the assumption that 
any $B$-module is $I_{\tilde{G}}(B)$-invariant.
Hence we can take the morphisms $(\alpha_1, \ldots, \alpha_{p^n}) \in \Hom_{\hmtB}(X, Z)^{\oplus p^n}$
corresponding to $\alpha \in \Hom_{\hmttildeB}(X\Ind, Z\Ind)$.
Now, the morphism $f: X\rightarrow M$ is a left $\mathcal{M}$-approximation of $X$
and $Z$ is an object of $\mathcal{M}$, so 
there exists a morphism $\beta_i$ in $\Hom_{\hmtB}(M, Z)$
such that $\alpha_i=\beta_i\circ f$ for each $1 \leq i \leq p^n$.
Hence we get $\alpha=\beta\circ f^{\Ind}$, where $\beta$ is a 
morphism in $\Hom_{\hmttildeB}(M\Ind, Z\Ind)$ 
corresponding to 
$(\beta_1, \ldots, \beta_{p^n}) \in \Hom_{\hmtB}(M, Z)^{\oplus p^n}$
which is isomorphic to
$\Hom_{\hmttildeB}(M\Ind, Z\Ind)$.
Therefore $f\Ind: X\Ind \rightarrow M\Ind$ is a left $\mathcal{M}\Ind$-approximation
of $X\Ind$ in $\hmttildeB$.
\end{proof}

\begin{prop}\label{commute}
For any tilting complex $T$ over $B$,
the induced complex $T\Ind$ is a tilting complex over $\tilde{B}$
and irreducible tilting mutations commute with the induction functor $(-)\Ind$.
\end{prop}

\begin{proof}
Let $T$ be a tilting complex over $B$.
By the assumption (Inv), Lemma \ref{Inv-lemma} means that
the complex $T$ is $I_{\tilde{G}}(B)$-invariant.
Hence applying \cite[Proposition 2.1]{Ma2005} to $T$, 
we have that the induced complex
$T^{\uparrow I_{\tilde{G}}(B)}$ is a tilting complex.
Moreover, by \cite[Theorem 5.5.12]{NT1989}
the induction functor
$ -\otimes_{kI_{\tilde{G}}(B)} k\tilde{G}$
gives Morita equivalence between ${\rm mod}\:A$ and ${\rm mod}\:\tilde{B}$, where $A$ is the unique block
of $kI_{\tilde{G}}(B)$ covering $B$.
Hence $(T^{\uparrow I_{\tilde{G}}(B)}){\Ind} = T{\Ind}$ 
is a tilting complex over $\tilde{B}$.

Next we show that, for a tilting complex $\mu^\epsilon_{i_k}(T)$,
it holds that $\big(\mu_{i_k}^\epsilon(T)\big)\Ind$
is isomorphic to $\mu^\epsilon_{i_k}(T\Ind)$,
where $\epsilon$ means $+$ or $-$.
We prove the case of $\epsilon$ being $-$
because the other case can be shown similarly.

For a decomposition $T = T_{i_k} \oplus (\bigoplus_{i\neq i_k}T_i ) $
of $T$ with $T_{i_k}$ indecomposable,
we consider the distinguished triangle
$$ T_{i_k} \xrightarrow{f} M \rightarrow Cone(f) \rightarrow,$$
where $f$ is a minimal left approximation of $T_{i_k}$ with respect 
to ${\rm add}(\bigoplus_{i\neq i_k}T_i)$.
Then, by Lemma \ref{approximation}, the morphism
$f\Ind : T_{i_k}\Ind \rightarrow M{\Ind}$ is a minimal left approximation
of $T_{i_k}\Ind$ with respect to ${\rm add}(\bigoplus_{i\neq i_k}T_i{\Ind})$.
Hence we get a distinguished triangle 
$$ T_{i_k}{\Ind} \xrightarrow{f\Ind} M{\Ind} \rightarrow Cone(f\Ind) \rightarrow$$
with a minimal left approximation $f\Ind$ of $T_{i_k}{\Ind}$
with respect to ${\rm add}(\bigoplus_{i\neq i_k}T_i{\Ind})$
by Lemma \ref{approximation}.
Moreover we can show that $Cone(f\Ind)$ is isomorphic to 
$Cone(f)\Ind$ easily (for example see \cite[Lemma 3.4.10]{Z2014}).
Therefore we get
$$ \mu^-_{i_k}(T{\Ind}) 
=
Cone(f\Ind) \oplus(\bigoplus_{i\neq i_k}T_i{\Ind})
\cong 
Cone(f)\Ind \oplus(\bigoplus_{i\neq i_k}T_i{\Ind})
=
\big(\mu^-_{i_k}(T)\big)\Ind,
$$
which means that mutations commute with the induction functor in
our situation.
\end{proof}

\begin{thm}\label{tilting-discrete}
Assume that the block $B$ is a tilting-discrete algebra
and that for any algebra $\Lambda$ derived equivalent to $B$
the number of two-term tilting complexes over $\Lambda$ is finite.
Then the block $\tilde{B}$ of $k\tilde{G}$
is also a tilting-discrete algebra.
\end{thm}

\begin{proof}
We can assume $I_{\tilde{G}}(B)=\tilde{G}$ (see the proof
of Proposition \ref{commute}).
We show that
the set $2$-${\rm tilt}_{\tilde{T}}\,\tilde{B}$ is a finite set
for any tilting complex $\tilde{T}$ given by
iterated irreducible left mutation from $\tilde{B}$,
which means the $\tilde{B}$ is a tilting-discrete algebra by
Theorem \ref{AiM2017}.

By the choice of the tilting complex $\tilde{T}$,
there exist a sequence 
$(i_k, i_{k-1}, \ldots, i_1)$
such that 
$$ \tilde{T} \cong (\mu_{i_{k}}^-\mu_{i_{k-1}}^-
\cdots \mu_{i_{1}}^-)(\tilde{B}).$$
Moreover, for a tilting complex
$T:=(\mu_{i_{k}}^-\mu_{i_{k-1}}^-
\cdots \mu_{i_{1}}^-)(B)$ over $B$,
we have $T{\Ind} \cong \tilde{T}$ by Proposition \ref{commute}.

On the other hand, for any tilting complex $U$ and for any $i>0$, by Frobenius reciprocity, Mackey's formula,
the assumption (Inv) and Lemma \ref{Inv-lemma}, we have the following
isomorphisms:
\begin{align*}
\Hom_{\hmttildeB}(T\Ind, U\Ind[i])
&\cong
\Hom_{\hmtB}(T, U\Ind{}_{\downarrow G}[i])\\
&\cong
\bigoplus_{\tilde{g}G\in \tilde{G}/G}\Hom_{\hmtB}(T,  U{\,\tilde{g}}[i])\\
&\cong
\Hom_{\hmtB}(T, U[i])^{\oplus p^n}.
\end{align*}
Hence we have that if $T\geq U$ then $T\Ind \geq U\Ind$.
By a similar argument, we have that
if $U\geq T[1]$ then $U\Ind \geq T\Ind[1]$.
Summarizing the above,
we have that if $T\geq U \geq T[1]$,
then $\tilde{T}\geq U{\Ind} \geq \tilde{T}[1]$.
Hence we get an injection map preserving partial order structures;
$$ 2\textnormal{-}{\rm tilt}_T\,B \rightarrowtail 
2\textnormal{-}{\rm tilt}_{\tilde{T}}\,\tilde{B}
\:\:(U \mapsto U{\Ind}).  $$
Now, both Hasse quivers of $2\textnormal{-}{\rm tilt}_T\,B$
and $2\textnormal{-}{\rm tilt}_{\tilde{T}}\,\tilde{B}$ are
$\vert B \vert$-regular graphs 
by Lemma \ref{graphlemma} (we remark that $\vert B \vert=\vert \tilde{B} \vert$
because $\tilde{G}/G$ is a $p$-group),
and the two sets have
the respective maximal elements $T$ and $\tilde{T}$
and have the respective minimal elements $T[1]$ and $\tilde{T}[1]$.
Moreover 
by Theorem \ref{Theorem 2.35} and Proposition \ref{commute},
for any tilting complexes $U$ and $V$ in $2\textnormal{-}{\rm tilt}_T\,B$
with $U>V$,
if there is no tilting complex $L$ over $B$ satisfying that
$U>L>V$,
then there is no tilting complex $\tilde{L}$ over $\tilde{B}$
satisfying that $U\Ind>\tilde{L}>V\Ind$.
Also, by the assumption and Lemma \ref{graphlemma},
the Hasse quiver of $2\textnormal{-}{\rm tilt}_T\,B$
is a finite  $|B|$-regular graph. 
Hence we get that the above injection map is an isomorphism
between the partially ordered sets.
Therefore the set $2\textnormal{-}{\rm tilt}_{\tilde{T}}\,\tilde{B}$
is a finite set.
\end{proof}

\begin{thm}\label{poset-iso}
Assume that the block $B$ is a tilting-discrete algebra
and that for any algebra $\Lambda$ derived equivalent to $B$
the number of two-term tilting complexes over $\Lambda$ is finite.
Then the induction functor $(-)\Ind : \tilt B \rightarrow \tilt \tilde{B}$
induces an isomorphism of partially ordered sets.
\end{thm}

\begin{proof}
By Proposition \ref{commute}, 
we have the well-defined map $(-)\Ind : \tilt B \rightarrow \tilt \tilde{B}$.
Let $T_1$ and $T_2$ be tilting complexes over $B$ and
assume that $T_1 > T_2$ and that there exists no tilting complex $L$ over $B$
such that $T_1 > L > T_2$.
Hence, by Theorem \ref{Theorem 2.35} and Proposition \ref{commute} again,
we have that 
$T_1\Ind > T_2\Ind$ and there exists no tilting complex $\tilde{L}$ over $\tilde{B}$
such that $T_1\Ind > \tilde{L} > T_2\Ind$ in $\tilt \tilde{B}$.
Also, it is obvious that the map $(-)\Ind : \tilt B \rightarrow \tilt \tilde{B}$
is an injection.
There remains to show that the map is surjection.
By Theorem \ref{tilting-discrete} and Theorem \ref{strongly-connected},
$\tilde{B}$ is a strongly tilting-connected algebra.
Hence, for any tilting complex $\tilde{T}$ over $\tilde{B}$, 
there exists a sequence 
$(i_k, i_{k-1}, \ldots, i_1)$
such that 
$ \tilde{T} \cong (\mu_{i_{k}}^\epsilon\mu_{i_{k-1}}^\epsilon
\cdots \mu_{i_{1}}^\epsilon)(\tilde{B}),$ where $\epsilon$
means $+$ or $-$.
Therefore, for a tilting complex $T=(\mu_{i_{k}}^\epsilon\mu_{i_{k-1}}^\epsilon
\cdots \mu_{i_{1}}^\epsilon)(B)$ over $B$,
by Proposition \ref{commute}, 
we have that $\tilde{T}\cong T{\Ind}$.
\end{proof}


As an immediate application of Theorem \ref{tilting-discrete},
we show that the homotopy category $\hmttildeB$ satisfies
a Bongartz-type Lemma.

\begin{thm}\label{Bongartz}
In the same notation as Theorem \ref{tilting-discrete},
any pretilting complex $\tilde{P}$ in $\hmttildeB$ is a partial tilting complex,
that is, there exists a pretilting complex $\tilde{P'}$ in $\hmttildeB$
such that $\tilde{P}\oplus \tilde{P'}$ is a tilting complex in $\hmttildeB$.
\end{thm}

\begin{proof}
By Theorem \ref{tilting-discrete},
$\tilde{B}$ is a tilting-discrete algebra,
that is, 
the homotopy category $\hmttildeB$ is a tilting-discrete category.
Therefore we have the result by \cite[Theorem 2.15]{AiM2017}.
\end{proof}


In the rest of this section,
we assume the block $B$ of $kG$ has
a cyclic defect group.
In this case, the block $B$ has the following nice properties.
\begin{prop}\label{cyclic-block}
Let $B$ be a block of $kG$ with a cyclic defect group.
Then we have the following.
\begin{enumerate}
\item
\cite[Section 5]{Al86}
The block $B$ is a Brauer tree algebra.
\item
\cite[Theorem 4.2]{Ri1989-2}
Any algebra derived equivalent to $B$ is a Brauer tree algebra
associated to a Brauer tree
with the same number of edges and multiplicity as the Brauer tree of $B$.
\item
\cite[Lemma 3.22]{KK2020}
The condition (Inv) holds automatically. 
\item
\cite[Theorem 1.2]{Ai2013}
$B$ is a tilting-discrete algebra.
\item
For any tilting complex $P$ over $B$,
the set $2$-${\rm tilt}_P\,B$ is a finite  $|B|$-regular graph
(see Corollary \ref{graphlemma-for-Brauer-tree-algebra}).
\end{enumerate}

\end{prop}

\begin{thm}\label{mainthm}
Let $G$ be a finite group and $\tilde{G}$ be a finite group having $G$
as a normal subgroup
and assume that the index of $G$ in $\tilde{G}$ is a $p$-power.
Let $B$ be a block of the finite group $G$ with cyclic defect group,
and $\tilde{B}$ be the unique block of $k\tilde{G}$ covering $B$.
Then the following hold.
\begin{enumerate}
\item
$\tilde{B}$ is a tilting-discrete algebra.
\item
The induction functor $(-)\Ind : \tilt B \rightarrow \tilt \tilde{B}$
induces an isomorphism of partially ordered sets.
\end{enumerate}
\end{thm}

\begin{proof}
By Proposition \ref{cyclic-block},
the blocks $B$ and $\tilde{B}$
satisfy the assumption in Theorem \ref{tilting-discrete}
and Theorem \ref{poset-iso}.
Therefore we have the result.
\end{proof}


\begin{thebibliography}{00}

\bibitem{AAC2018}
T.~Adachi, T.~Aihara, A.~Chan,
{\it Classification of two-term tilting complexes over Brauer graph algebras.}
Math. Z. {\bf 290} (2018), no. 1--2, 1--36.
\bibitem{AIR2013}
T.~Adachi, O.~Iyama, I.~Reiten,
{\it $\tau$-tilting theory.}
 Compos. Math. {\bf 150} (2014), no. 3, 415--452.

   
\bibitem{AI2012}
           T.~Aihara, O.~Iyama, 
           {\it Silting mutation in triangulated categories.} 
           J. Lond. Math. Soc. (2) {\bf85} (2012), no. 3, 633--668.

\bibitem{Ai2013}
          T.~Aihara,
          {\it Tilting-connected symmetric algebras.} 
          Algebr. Represent. Theory {\bf16} (2013), no. 3, 873--894. 
          
\bibitem{AiM2017}
T.~Aihara, Y.~Mizuno
{\it Classifying tilting complexes over preprojective algebras of Dynkin type.}
Algebra Number Theory {\bf11} (2017), no. 6, 1287--1315.

\bibitem{Al86}
           J.~L.~Alperin, 
           {\it Local representation theory}, Cambridge University Press, 1986.
          
\bibitem{AMN2020}
H.~Asashiba, Y.~Mizuno, K.~Nakashima,
{\it Simplicial complexes and tilting theory for Brauer tree algebras.}
J. Algebra {\bf551} (2020), 119--153.

\bibitem{B1990}
M.~Brou\'{e}, {\it Isom\'{e}tries parfaites,
types de blocs, cat\'{e}gories d\'{e}riv\'{e}es,} Ast\'{e}risque
{\bf 181}--{\bf 182} (1990), 61--92.

\bibitem{B1994}
M.~Brou\'{e}, 
{\it Equivalences of blocks of group algebras.}
in: Finite-dimensional algebras and related topics,
Ottawa, ON, 1992, in: NATO Adv. Sci. Inst. Ser. C Math. Phys. Sci.,
vol. 424, Kluwer Acad. Publ., Dordrecht, 1994, pp. 1--26.

\bibitem{BPP2016}
N.~Broomhead, D.~Pauksztello, D.~Ploog,
{\it Discrete derived categories II: the silting pairs CW complex 
and the stability manifold.}
J. Lond. Math. Soc. (2) {\bf 93} (2016), no. 2, 273--300.

 
\bibitem{E1990}
K.~Erdmann,
{\it Blocks of tame representation type and related algebras.}
Lecture Notes in Mathematics, 1428, Springer-Verlag, Berlin, 1990.

\bibitem{EJR2018}
F.~Eisele, G.~Janssens, T.~Raedschelders, 
{\it A reduction theorem for $\tau$-rigid modules.}
Math. Z. {\bf 290} (2018), no. 3--4, 1377--1413.
 
 
\bibitem{G1959}
           J.~A.~Green,
           {\it On the indecomposable representations of a finite group.}
           Math. Z. 70 (1958/59), 430--445. 
           
\bibitem{HL2000}
M.~E.~Harris, M.~Linckelmann,
{\it
Splendid derived equivalences for blocks of finite p-solvable groups.}
J. London Math. Soc. (2) {\bf 62} (2000), no. 1, 85--96. 
          
\bibitem{KK2020}
R.~Koshio, Y.~Kozakai,
{\it On support $\tau$-tilting modules over blocks covering cyclic blocks.}
J. Algebra {\bf 580} (2021), 84--103.

\bibitem{Koz2019}
Y.~Kozakai, 
{\it Foldings and two-sided tilting complexes for Brauer tree algebras.}
Osaka J. Math. {\bf 56} (2019), no. 1, 133--164.

\bibitem{Koz-Ku2018}
Y.~Kozakai, N.~Kunugi, 
{\it Two-sided tilting complexes for Brauer tree algebras.} 
J. Algebra Appl. {\bf 17} (2018), no. 12, 1850231, 26 pp.


\bibitem{Ma2005}
A.~Marcus,
{\it Tilting complexes for group graded algebras, II.}
Osaka J. Math. {\bf 42} (2005), no. 2, 453--462.

\bibitem{NT1989}
H.~Nagao, Y.~Tsushima,
{\it Representations of finite groups,}
Academic Press, Inc., Boston, MA, 1989, Translated from the Japanese.

\bibitem{Ri1989-1}
          J.~Rickard,
          {\it Morita theory for derived categories}, 
          J.~London Math. Soc.(2), {\bf39} (1989), 436--456.


\bibitem{Ri1989-2}
          J.~Rickard, 
          {\it Derived categories and stable equivalence}, J. Pure Appl. Algebra  61  (1989),  no. 3, 303-317.

\bibitem{Ri1996}
J.~Rickard,
{\it Splendid equivalences: derived categories and permutation modules.}
Proc. London Math. Soc. (3) {\bf 72} (1996), no. 2, 331--358.

\bibitem{1685Rickard} 
           J.~Rickard, 
           {\it Triangulated Categories in the Modular
           Representation Theory of Finite Groups,}
           in: S. K\"onig, A. Zimmermann (Eds.), Derived Equivalences for Group Rings,
           in: Lecture Notes in Math., vol. 1685, 1998, pp. 177--198. 

\bibitem{RS2002}
J.~Rickard, M.~Schaps, 
{\it Folded tilting complexes for Brauer tree algebras}, 
Adv. Math. {\bf 171} (2002),  no. 2, 169--182.

\bibitem{Rou1993}
          R.~Rouquier,
          {\it From stable equivalences to Rickard equivalences for blocks with cyclic defect}. 
          Groups '93 Galway/St. Andrews, Vol. 2,  pp. 512-523, 
          London Math. Soc. Lecture Note Ser., 212, Cambridge Univ. Press, Cambridge, 1995.  


\bibitem{HKK2010}
          M.~Holloway, S.~Koshitani, N.~Kunugi,
          {\it Blocks with nonabelian defect groups which have cyclic subgroups of index $p$.} 
          Arch. Math. (Basel) {\bf 94} (2010), no. 2, 101--116. 

  \bibitem{Z2014}
         A.~Zimmermann,
         {\it Representation theory,} 
         A homological algebra point of view, 
         Algebra and Applications, {\bf19}. Springer, Cham, 2014.
          


  
 %
%
%
%
  

%
%
%
%
%

%

%
%
%
%
%
%
%

%
%

\end{thebibliography}
\end{document}